\documentclass{amsart}[11pt]
\usepackage{amssymb}
\usepackage{latexsym}
\usepackage[all]{xy}

\newtheorem{theorem}{Theorem}[section]
\newtheorem{lemma}[theorem]{Lemma}
\newtheorem{proposition}[theorem]{Proposition}
\newtheorem{corollary}[theorem]{Corollary}
\theoremstyle{definition}
\newtheorem{definition}[theorem]{Definition}
\newtheorem{example}[theorem]{Example}

\newtheorem{remark}[theorem]{Remark}
\numberwithin{equation}{theorem}
\def\Chi{{\mathbb X}}
\def\Ker{{\rm Ker}}

\def\div{{\rm div}}

\def\rk{{\rm rk}}

\def\KK{{\mathbb K}}

\def\ZZ{{\mathbb Z}}

\def\PP{{\mathbb P}}
\def\ZZ{{\mathbb Z}}
\def\Of{{\mathcal{O}}}

\def\Rf{{\mathcal{R}}}

\def\Ff{{\mathcal{F}}}

\def\WDiv{\operatorname{WDiv}}

\def\Quot{\operatorname{Quot}}

\def\SL{\operatorname{SL}}

\def\Cl{\operatorname{Cl}}
\def\Spec{{\rm Spec}}
\def\Pic{{\rm Pic}}
\newcounter{itemnumber}

\begin{document}

\sloppy

\title[On factoriality of Cox rings]
{On factoriality of Cox rings}
\author[I.~Arzhantsev]{Ivan~V. Arzhantsev} 
\thanks{The author gratefully acknowledges the support from INTAS YS 05-109-4958 and Deligne 2004 Balzan prize in mathematics}
\address{Department of Higher Algebra, 
Faculty of Mechanics and Mathematics, 
Moscow State University,
Leninskie Gory, GSP-2, Moscow, 119992, Russia}
\email{arjantse@mccme.ru}
\begin{abstract}
Generalized Cox's construction associates with an algebraic variety a remarkable invariant -- 
its total coordinate ring, or Cox ring. In this note we give a new proof of factoriality of the Cox ring 
when the divisor class group of the variety is finitely generated and free. The proof is based 
on a notion of graded factoriality. We show that if the divisor class group has torsion,
then the Cox ring is again factorially graded, but factoriality may be lost. 
\end{abstract}

\maketitle


\section{Introduction}

Let $X$ be an irreducible normal algebraic variety over an algebraically closed field $\KK$
with a free finitely generated divisor class group $\Cl(X)$.
Denote by $\WDiv(X)$
the group of Weil divisors on $X$ and fix a sublattice $K\subset \WDiv(X)$ 
which maps onto $\Cl(X)$ isomorphically. Following famous D.~Cox's construction~\cite{Cox} 
from toric geometry, define the Cox ring of the variety $X$ as
$$
\ \ R(X)=\bigoplus_{D\in K} \Of(X,D), \ \ \ \text{where} \ \ \ \Of(X,D)=\{\, f\in\KK(X) \,\mid\, \div(f)+D\ge 0\,\}.
$$ 
Multiplication on graded components of $R(X)$ coincides with multiplication in the field $\KK(X)$
of rational functions, and extends to other elements by distributivity. It can be easily  
checked that the ring $R(X)$ depends on the choice of the lattice $K$ only up to isomorphism
(for a more general statement, see Proposition~\ref{prop1}).   
An important property of $R(X)$ is that it is a factorial ring,
see~\cite{BH1}, \cite{EKW}. Here we give a new proof of this result.
 
\begin{theorem}\label{t1}
The ring $R(X)$ is factorial.
\end{theorem}

Our main aim is to show that factoriality of $R(X)$ reflects the fact that any effective Weil divisor on $X$ is a unique
non-negative integral combination of prime divisors. This observation immediately implies that the multigraded ring 
$R(X)$ is "factorial on the set of homogeneous elements" (we call this a priori weaker property "graded factoriality"). Further, 
we prove that factoriality follows from graded factoriality. This approach is realized in Section~\ref{sec1}.

In Section~\ref{sec2}, the Cox ring $R(X)$ is defined in the case when the divisor class group 
has torsion. Following \cite{BH1} and \cite{Ha}, we check that $R(X)$ is well-defined. 
Here the ring $R(X)$ is also factorially graded, but factoriality may be lost.
The corresponding examples are given in Section~\ref{sec3}, where we describe the Cox ring
of a homogeneous space of an affine algebraic group.

 The auhtor is grateful to J.~Hausen for valuable discussions which led
to some of the results given below. 


\section{Proof of Theorem~\ref{t1}} \label{sec1}

We start with some elementary properties of multigraded algebras. Let $R$ be a commutative associative
algebra with unit over a field $\KK$. Assume that $R$ is graded by the lattice $\ZZ^n$,
$$
R=\bigoplus_{u\in\ZZ^n} R_u.
$$
Denote by $R^{\times}$ (resp. $R^+$) 
the multiplicative semigroup of invertible (resp. homogeneous) elements of $R$. 

\begin{lemma} \label{lem1}
\begin{enumerate}
\item Suppose that for any $a,b\in R^+$ the condition $ab=0$ implies $a=0$ or $b=0$.
Then $R$ has no zero-divisors.
\item If $R$ has no zero-divisors and for any $a,b\in R^+$ the condition $ab=1$ implies $a,b\in R_0$,
then $R^{\times}=R_0^{\times}$. 
\item Suppose that $R$ has no zero-divisors. If $a\in R^+$ and $a=bc$, then $b,c\in R^+$.
\end{enumerate}
\end{lemma} 

\begin{proof}
Let us fix the lexicographic order on the lattice $\ZZ^n$. With any element $a\in R$
one associates two homogeneous elements $L(a)$ and $l(a)$, namely its leading and 
lowest terms. Clearly, $L(ab)=L(a)L(b)$ and $l(ab)=l(a)l(b)$. Now the statements of the lemma
follow easily.
\end{proof}

\begin{corollary}
\begin{enumerate}
\item The Cox ring $R(X)$ has no zero-divisors.
\item The semigroup $R(X)^{\times}$ coincides with $\Of(X)^{\times}$, 
where $\Of(X)$ is the algebra of regular functions on the variety $X$.
\end{enumerate}
\end{corollary}

\begin{proof}
Statement (i) follows from Lemma~\ref{lem1}, (i), because $\KK(X)$ has no zero-divisors.
To prove (ii), note that if $f_1\in\Of(X,D_1)$, $f_2\in\Of(X,D_2)$ and
$f_1f_2=1$, then $D_1+D_2=0$, $\div(f_1)+\div(f_2)=0$, and the sum of the effective divisors
$\div(f_1)+D_1$ and $\div(f_2)+D_2$ equals zero. Therefore $\div(f_i)+D_i=0$, $i=1,2$, and 
$D_i=0$. Now one uses Lemma~\ref{lem1}, (ii) and the equality $R(X)_0=\Of(X)$.
\end{proof}

\begin{definition}
Let $A$ be a finitely generated abelian group, and $R=\oplus_{u\in A} R_u$ be an $A$-graded
algebra. 
\begin{itemize}
\item A non-zero element $a\in R^+\setminus R^{\times}$ is called {\it h-irreducible}, if 
the condition $a=bc$, $b,c\in R^+$ implies that either $b$ or $c$ is invertible.
\item An $A$-graded algebra $R$ is said to be {\it factorially graded}, if any its non-zero
non-invertible homogeneous element may be expressed as a product of h-irreducible elements,
and such an expression is unique up to association and renumbering. 
\end{itemize}
\end{definition}

\begin{remark}
Assume that $R=\oplus_{u\in\ZZ^n} R_u$ has no zero-divisors. It follows from Lemma~\ref{lem1} (iii)
that if $R$ is factorial, then it is factorially graded. 
\end{remark}

\begin{proposition} \label{pr}
The Cox ring $R(X)=\bigoplus_{D\in K} \Of(X,D)$ is factorially graded.
\end{proposition}

\begin{proof}
Effective Weil divisors on $X$ are in one-to-one correspondence with classes
of associated elements of $R(X)^+$, and the product of homogeneous elements corresponds to the sum
of divisors. This shows that classes of h-irreducible elements of the ring $R(X)$ 
correspond to prime divisors. Since any effective Weil divisor is a unique non-negative
integral combination of prime divisors, the ring $R(X)$ is factorially graded.
\end{proof}

Below we shall need the following well-known lemma. For convenience of the reader we give a short
proof, cf.~\cite[Prop.~17.1]{Ti}.

\begin{lemma}\label{l2}
Let $T$ be an algebraic torus and $Z$ a normal algebraic variety with a regular
$T$-action. Then any Weil divisor on $Z$ is linearly equivalent to a $T$-invariant
Weil divisor.
\end{lemma}

\begin{proof}
By normality, the singular locus of $Z$ has codimension $\ge 2$, and one may assume that $Z$ is smooth. 
Then any Weil divisor $D$ on $Z$ is Cartier. We may assume that $D$ is effective. 
There exists a $T$-linearization of the corresponding
line bundle, which defines a structure of a rational $T$-module on
$\Of(Z,D)$~\cite[2.4]{KKLV}. Any $T$-eigenvector in $\Of(Z,D)$ represents a $T$-invariant divisor 
equivalent to $D$.
\end{proof}

Again by normality, the passage from $X$ to its smooth locus does not change the Cox ring.
Further we shall assume that $X$ is smooth. Following~\cite{BH1}, let us introduce the "universal torsor"
$\hat X\to X$ over $X$. Consider the $K$-graded sheaf of $\Of_X$-algebras 
$$
\Rf_X=\oplus_{D\in K} \Of(D), \ \ \ \text{where} \ \ \ \Of(U,D)=\{\, f\in\KK(X) \,\mid\, [\div(f)+D]|_U\ge 0\,\},
$$
and the relative spectrum $\widehat{X}=\Spec_X(\Rf_X)$ of this sheaf over $X$. 
Clearly, $R(X)=H^0(X,\Rf_X)=\Of(\widehat{X})$. The $K$-grading on $\Rf_X$
defines a regular action of the torus $T=\Spec(\KK[K])$ on $\widehat{X}$, and the canonical 
affine morphism $p:\widehat{X}\to X$ given by inclusion $\Of_X\subset\Rf_X$ is
$T$-invariant. Since all divisors on $X$ are Cartier,
$p$ is a locally trivial fibration with $T$ as a fiber. In particular, $\widehat{X}$ is smooth.
Fix an open affine covering
$\{U_i\}$ of $X$. Each divisor $D_i=X\setminus U_i$ corresponds to an element
$f_i\in R(X)^+$. The section $f_i$ of the sheaf $\Rf_X$ is invertible exactly over $U_i$.
Therefore $p^{-1}(U_i)$ coincides with $\widehat{X}_{f_i}=\{z\in\widehat{X} : f_i(z)\ne 0\}$.
We get an open affine covering $\{\widehat{X}_{f_i}\}$ of $\widehat{X}$,
where $f_i\in\Of(\widehat{X})$. It is easy to deduce from this that $\widehat{X}$ is a quasi-affine variety,
see.~\cite[Ch.~2, App., Lemma~8]{Gr}.

\smallskip

We are ready to finish with the proof of Theorem~\ref{t1}. Assume that the ring $R(X)$ 
is not factorial. Then $R(X)$ contains a non-principal prime ideal of height one. 
Since $R(X)=\Of(\widehat{X})$, it is a Krull ring, and its prime ideals of height one are in bijection 
with essential discrete valuations of the ring $R(X)$, see
\cite[I.3]{Sa}. On the other hand, prime divisors on a normal quasi-affine variety are in bijection
with essential discrete valuations of the ring of regular functions.
Consequently, there is a non-principal divisor on $\widehat{X}$.
By Lemma~\ref{l2}, one may assume that there is a prime $T$-invariant non-principal divisor on $\widehat{X}$.
The corresponding ideal $I\lhd R(X)$ contains a homogeneous, and even an h-irreducible element.
Suppose that for some $a,b\in R(X)$ the product
$ab$ is divisible by $p$, but any homogeneous component of both $a$ and $b$ is not divisible by $p$.
Since $R(X)$ is factorially graded, one considers the product of leading components of $a$ and $b$ and comes to 
contradiction. Thus the ideal $(p)\lhd R(X)$ is prime. The inclusion $(p)\subseteq I$
implies the equality $(p)=I$. This contradiction completes the proof of Theorem~\ref{t1}. 

\medskip

We finish this section with the following observation.

\begin{proposition} \label{pr9}
Let $R=\oplus_{u\in\ZZ^n} R_u$ be a multigraded finitely generated $\KK$-algebra without
zero divisors. Assume that $R$ is factorially graded. Then $R$ is factorial.
\end{proposition}

\begin{proof}
Let us show that $R$ integrally closed. The multigrading defines actions of an $n$-dimensional torus
$T$ on the algebra $R$ and its quotient field $\Quot(R)$ by automorphisms. It is known that the integral closure
$\overline{R}$ of $R$ in $\Quot(R)$ is a $T$-invariant subalgebra, and the $T$-action defines a structure of a 
rational $T$-module on $\overline{R}$. In particular, $\overline{R}$ is $\ZZ^n$-graded, and $R$ is its
homogeneous subalgebra. Any homogeneous element 
$r\in\overline{R}$ may be expressed as $r=\frac{r_1}{r_2}$, $r_1,r_2\in R^+$.   
Since $R(X)$ is factorially graded and $\frac{r_1}{r_2}$ is integral over $R$, 
one has $\frac{r_1}{r_2}\in R^+$ and $\overline{R}=R$.

 It is well known that $R$ is factorial if and only if $\Cl(Z)=0$ for the normal affine variety 
$Z=\Spec(R)$. The arguments given above show that any prime $T$-invariant divisor on $Z$ is principal, 
and Lemma~\ref{l2} completes the proof.
\end{proof}

\begin{remark}
As in the proof of Proposition~\ref{pr9}, one may show that in characteristic zero 
a factorially graded finitely generated  algebra $R=\oplus_{u\in A} R_u$ 
without zero divisors is integrally closed
for any finitely generated abelian group $A$. 
\end{remark}


\section{Torsion in the divisor class group} \label{sec2}

In this section we define the Cox ring $R(X)$ for a variety $X$ with arbitrary finitely generated divisor class group 
and check that $R(X)$ is well-defined, compare \cite{BH1}, \cite{Ha}.

Let $S\subset\WDiv(X)$ be a finitely generated subgroup that projects to $\Cl(X)$ surjectively.
Consider a ring 
$$
T_S(X)=\bigoplus_{D\in S} \Of(X, D).
$$ 
Let $S^0\subset S$ be the kernel of the projection $S\to\Cl(X)$.
Take compatible bases $D_1,\dots, D_s$ in $S$ and $D_1^0=d_1D_1,\dots,D_r^0=d_rD_r$ in $S^0$, $r\le s$. 

\smallskip

We call a family of rational functions $\Ff=\{F_D\in\KK(X)^{\times} : D\in S^0\}$ {\it coherent},
if $\div(F_D)=D$ and $F_{D+D'}=F_DF_{D'}$ (the term "shifting family" was used in~\cite{BH1} for a similar notion). Obviously, the family   
$\Ff$ is defined by $F_{D_i^0}$, $i=1,\dots,r$: if $D=a_1D_1^0+\dots+a_rD_r^0$,
then $F_D=F_{D_1^0}^{a_1}\dots F_{D_r^0}^{a_r}$. Let us fix a coherent family $\Ff$.

Let $D_1,D_2\in S$ and $D_1-D_2\in S^0$. A map $f\to F_{D_1-D_2}f$ is an isomorphism between vector spaces 
$\Of(X,D_1)$ and $\Of(X,D_2)$. One easily checks, that the linear span of elements
$f-F_{D_1-D_2}f$ over all $D_1,D_2$ with $D_1-D_2\in S^0$ and all $f\in\Of(X,D_1)$ is an ideal $I(S,\Ff)$ of $T_S(X)$.
Define the Cox ring of the variety $X$ as 
$$
R_{S,\Ff}(X)=T_S(X)/I(S,\Ff).
$$

\smallskip

Since $D/D^0\cong\Cl(X)$, the ring $R_{S,\Ff}(X)$ carries a natural $\Cl(X)$-grading.

\begin{lemma}
Assume that $\Of(X)^{\times}=\KK^{\times}$. Then the ring $R_{S,\Ff}(X)$
does not depend on a choice of $\Ff$ up to isomorphism.
\end{lemma}

\begin{proof}
The set of functions $F_{D_1^0},\dots,F_{D_r^0}$ is defined up to transformations $F_{D_i^0}\to \gamma_iF_{D_i^0}$,
$\gamma_i\in\KK^{\times}$. Fix elements $\alpha_i\in\KK^{\times}$, $i=1,\dots,r$, such that $\alpha_i^{d_i}=\gamma_i$,
and set $\alpha_{r+1}=\dots=\alpha_s=1$.  
Then the desired isomorphism of the quotient-rings is induced by an automorphism $T_S(X)\to T_S(X)$, which acts on the component
$\Of(X,D)$, $D=a_1D_1+\dots+a_sD_s$, via multiplication by $\alpha_1^{a_1}\dots\alpha_s^{a_s}$.
\end{proof}

\begin{proposition} \label{prop1}
Assume that $\Of(X)^{\times}=\KK^{\times}$. Then the ring $R_S(X)$ does not depend on a choice 
of $S$ up to isomorphism.
\end{proposition}

\begin{proof}
Let $M\subset\WDiv(X)$ be another finitely generated subgroup that projects to 
$\Cl(X)$ surjectively. One may assume that $\rk(S)\ge\rk(M)$. 

\begin{lemma}
There is a surjective homomorphism $\phi:S\to M$ such that the following
diagram is commutative:
$$ 
\xymatrix{
& 
S \ar[dr] \ar[rr]^{\phi} & & M \ar[dl]
&
\\
& &
\Cl(X)
& &
}
$$
\end{lemma}

\begin{proof}
Take compatible bases $D_1,\dots,D_s$ in $S$ and $d_1D_1,\dots,d_rD_r$ in $S^0$,
and also $M_1,\dots M_k$ in $M$ and $m_1M_1,\dots,m_pM_p$ in $M^0$, such that $d_i$ (resp. $m_i$)
is divisible by $d_{i+1}$ (resp. by $m_{i+1}$). The condition $S/S^0\cong\Cl(X)\cong M/M^0$ implies that 
$s-r=k-p$, and the sets $(d_1,\dots,d_r)$ and $(m_1,\dots,m_p)$ may differ only in the final group of units. 
It remains to put $\phi(D_i)=M_i$ for $i=1,\dots,p$, $\phi(D_i)=M_{i-r+p}$ for $i=r+1,\dots,s$, and $\phi(D_i)=0$
for other $i$.
\end{proof}

We return to the proof of Proposition~\ref{prop1}. Fix a coherent family $\Ff$ for the subgroup $(S+M)^0\subset S+M$,
and define a surjective homomorphism $\Phi: T_S(X)\to T_M(X)$ that sends a homogeneous component
$\Of(X,D)$ to $\Of(X,\phi(D))$ as $\Phi(f)=F_{D-\phi(D)}f$. The kernel of $\Phi$ is contained in
the ideal $I(S,\Ff|_{S^0})$, and the ideal $I(S,\Ff|_{S^0})$ itself maps surjectively to $I(M,\Ff|_{M^0})$.
This shows that $\Phi$ defines a homomorphism $R_S(X)\to R_M(X)$ that, in fact, is an isomorphism.
\end{proof}

Note that homogeneous elements of the ring $T_S(X)$ from different components corresponding to $D_1$ and $D_2$ with
$D_1-D_2\in S^0$ may define the same effective divisor on $X$. However, after factorization by the ideal
$I(S,\Ff)$ effective Weil divisors on $X$ are again in bijection with  
association classes in $R(X)^+$.

\begin{proposition} \label{pr7}
The ring $R(X)$ is factorially graded.
\end{proposition}

\begin{proof}
The statement follows from the proof of Proposition~\ref{pr}.
\end{proof}

\begin{remark}
By \cite[Cor.~1.2]{EKW}, the ring $T_S(X)$ is factorial. As we shall see in the next section,
this property may be lost for the ring $R(X)$ if the group $\Cl(X)$ has torsion. 
\end{remark} 


\section{Homogeneous spaces of an algebraic group} \label{sec3}

In this section we assume that the ground field $\KK$ has characteristic zero.
Let $G$ be a connected affine algebraic group with $\Cl(G)=0$ and without nontrivial characters.
Note that the first condition may be achieved by passing to a finite covering of a given group $G$ 
\cite[Prop.~4.6]{KKLV}. Denote by $\Chi(F)$ the group of characters of an algebraic group $F$.  
By Rosenlicht's Theorem, the condition $\Chi(G)=0$ is equivalent to
$\Of(G)^{\times}=\KK^{\times}$. 

 Let $H$ be a closed subgroup of $G$. The homogeneous space $G/H$ admits a canonical structure
of a smooth quasi-projective algebraic $G$-variety. In \cite{Po}, it was proved that
$\Cl(G/H)\cong\Pic(G/H)\cong\Chi(H)$. Let us recall how to establish the last isomorphism.

Any character $\chi\in\Chi(H)$ defines a one-dimensional $H$-module $\KK_{\chi}$.
Consider a homogeneous fiber bundle 
$$
L_{\chi}=G\times_H\KK_{\chi}:=(G\times\KK_{\chi})/H, \ \ \ h\cdot(g,a):=(gh^{-1},\chi(h)a).
$$
It is shown in~\cite{Po} that the projection $L_{\chi}\to G/H$ is a $G$-linearized
line bundle over $G/H$, and $L_{\chi_1}\otimes L_{\chi_2}\cong L_{\chi_1+\chi_2}$. 
Moreover, the map $\chi\to L_{\chi}$ defines an isomorphism between $\Chi(H)$ and $\Pic(G/H)$.

 Since $\Cl(G)=0$, the pull-back of the line bundle $L_{\chi}$ with respect to the projection $G\to G/H$
is a trivial $G$-linearized line bundle on $G$. This allows to identify the space of  
sections $H^0(X,L_{\chi})$ with the following subspace of $\Of(G)$:
$$
\Of(G)^{(H)}_{\chi}:=\{f\in\Of(G) : f(gh^{-1})=\chi(h)f(g) \ \text{for} \ \text{all} \ h\in H, g\in G\}.
$$
The tensor product of sections corresponds to the product in $\Of(G)$. 

\smallskip

Consider a subgroup $H_1=\cap_{\chi\in\Chi(H)}\Ker(\chi)$. The next theorem delivers an effective 
description of the Cox ring of the homogeneous space $G/H$, see also \cite[Lemma~3.14]{AH}.

\begin{theorem} \label{t2}
Let $G$ be a connected affine algebraic group with
$\Chi(G)=0$ and $\Cl(X)=0$, and $H$ be a closed subgroup of $G$. Then
$$
R(G/H)\cong\Of(G/H_1).
$$
\end{theorem}
 
\begin{proof}
The diagonalizable group $Q=H/H_1$ acts on $G/H_1$ 
by right multiplication. Since $\Chi(H)=\Chi(Q)$ and any rational $Q$-module is a direct sum of
one-dimensional submodules, we get $\Of(G/H_1)=\oplus_{\chi\in\Chi(H)}\Of(G)^{(H)}_{\chi}$. 

Weil divisors on $G/H$ are in bijection with lines generated by $Q$-semiinvariants $f\in\Quot(\Of(G/H_1))\subseteq\KK(G)$.
Further, effective divisors correspond to semiinvariants $f\in\Of(G/H_1)$.
Let us choose a multiplicative finitely generated group of semiinvariants in $\Quot(\Of(G/H_1))$ whose weights run through
the whole group $\Chi(Q)$. One may identify this subgroup with $S\subset\WDiv(G/H)$ and check easily that
the Cox ring associated with $S$ is isomorphic to $\oplus_{\chi\in\Chi(H)}\Of(G)^{(H)}_{\chi}\subset\Of(G)$.
\end{proof}

For a connected $H$, the group $Q$ is a torus and $\Chi(H)$ is free. Moreover, in this case $\Chi(H_1)=0$, thus
$\Cl(G/H_1)=0$ and the ring $R(G/H)$ is factorial (it follows also from Theorem~\ref{t1}).

For a disconnected $H$, the character group $\Chi(H_1)$ may be non-trivial.

\begin{example} \label{ex}
Let $G=\SL(2)$ and $H$ be the normalizer $N$ of a maximal torus $T\subset\SL(2)$.
Here $\Chi(H)$ is isomorphic to the cyclic group $\ZZ_2$ of order 2, $H_1=T$, and $\Chi(H_1)\cong\ZZ$. 
This shows that the Cox ring $R(\SL(2)/N)\cong\Of(\SL(2)/T)$ is not factorial. The space
$\SL(2)/N$ is a smooth affine surface $X$ with $\Cl(X)\cong\ZZ_2$, the ring $R(X)$ is isomorphic to
$$
\KK[x_1,x_2,x_3]/\langle x_2^2-x_1x_3-1\rangle,
$$
and the $\ZZ_2$-grading on $R(X)$ is given by $\deg(x_1)=\deg(x_2)=\deg(x_3)=\overline{1}$. 

\smallskip

One may propose several ways to associate a factorial ring with the surface $X$. Firstly, 
consider a subgroup $S\subset\WDiv(X)$ generated by a non-principal divisor. 
The ring $T_S(X)$ is isomorphic to 
$$
(\KK[x_1,x_2,x_3,t,t^{-1}]/\langle x_2^2-x_1x_3-1\rangle)^{\ZZ_2},
$$
where $\ZZ_2$ acts on the variables $x_1, x_2, x_3, t, t^{-1}$ via multiplication by $-1$.

Secondly, the space $\SL(2)/N$ admits a wonderful (the term is due to D.~Luna) $\SL(2)$-equivariant
embedding in $\PP^2$, and $R(\PP^2)\cong\KK[x_1,x_2,x_3]$. The Cox ring of a wonderful embedding of any spherical 
homogeneous space is described in~\cite{Br}.
\end{example}

\begin{example}
Using the construction of Example~\ref{ex}, one may find 
a smooth affine variety $X$ with $\Cl(X)\cong A$ and a non-factorial
Cox ring for any non-free finitely generated abelian group $A$.
Indeed, let $A\cong\ZZ_{d_1}\oplus\dots\oplus\ZZ_{d_s}\oplus\ZZ^n$.  
Put $X=G/H$, where $G$ is a direct product of $s+n$ simple groups,
$$
G=\SL(d_1)\times\dots\times\SL(d_s)\times\SL(2)\times\dots\times\SL(2),
$$
and $H=H(1)\times\dots\times H(s)\times T\times\dots\times T$, where $H(i)$ is an extension
of a maximal torus of the group $\SL(d_i)$ by elements of its normalizer that act as
degrees of one $d_i$-cycle. In this case,
$\Chi(H)\cong A$ and $\Chi(H_1)\cong\ZZ^{d_1+\dots+d_s-s}$. 
\end{example}



\begin{thebibliography}{}
\bibitem{AH} I.V.~Arzhantsev, J.~Hausen,
On embeddings of homogeneous spaces with small boundary.
J.~Algebra {\bf 304:2} (2006), 950--988 
%
\bibitem{BH1} F.~Berchtold, J.~Hausen, 
Homogeneous coordinates for algebraic varieties.
J.~Algebra {\bf 266:2} (2003), 636--670
%
\bibitem{Br} M.~Brion, The total coordinate ring of
a wonderful variety. J.~Algebra {\bf 313} (2007), 61--99
%
\bibitem{Cox} D.A.~Cox,
The homogeneous coordinate ring of a toric variety.
J. Alg. Geom. {\bf 4} (1995), 17--50 
%
\bibitem{EKW} J.~Elizondo, K.~Kurano, K.~Watanabe,
The total coordinate ring of a normal projective variety. 
J.~Algebra {\bf 276} (2004), 625--637
%
\bibitem{Gr} 
F.D.~Grosshans, Algebraic Homogeneous Spaces 
and Invariant Theory, LNM {\bf 1673}, Springer-Verlag Berlin, 1997
%
\bibitem{Ha} J.~Hausen, Cox rings and combinatorics II,
preprint {\tt arXiv:0801.3995}
%
\bibitem{KKLV} F.~Knop, H.~Kraft, D.~Luna, Th.~Vust, 
Local properties of algebraic group actions.
In: Algebraische Transformationsgruppen und Invariantentheorie, 
DMV Seminar~{\bf 13}, Birkh\"auser, Basel (1989), 63--76
%
%
\bibitem{Po}
V.L.~Popov, Picard group of homogeneous spaces of linear algebraic 
groups and one-dimensional homogeneous vector bundles. 
Izvestia Acad. Nauk SSSR, Ser. Mat. {\bf 38:2} (1974), 294--322; 
English transl.: Math. USSR-Izv. {\bf 8:2} (1974), 301--327
%
\bibitem{Sa} P.~Samuel, Lectures on unique factorization
domains, Tata Inst. Fund. Res., Bombay, 1964 
%
\bibitem{Ti} D.~A.~Timashev,  Homogeneous spaces and
equivariant embeddings, preprint {\tt math.AG/0602228}
%
\end{thebibliography}
\end{document}